\documentclass[12pt]{amsart}
\usepackage{amsmath,amsfonts,url}

\newcommand{\nexteq}{\displaybreak[0]\\ &=}
\newcommand{\N}{\mathbb{N}}
\newcommand{\bR}{\bar{R}}
\newcommand{\cF}{\mathcal{F}}
\newcommand{\cG}{\mathcal{G}}
\newcommand{\cR}{\mathcal{R}}

\newtheorem{lemma}{Lemma}
\newtheorem{theorem}[lemma]{Theorem}

\begin{document}

\title{Complementary Ramsey numbers and Ramsey graphs}
\author{Akihiro Munemasa
\and Masashi Shinohara}
\address{Tohoku University, Sendai, 980-8579, Japan}
\email{munemasa@math.is.tohoku.ac.jp}
\address{Shiga University, Shiga, 520-0862, Japan}
\email{shino@edu.shiga-u.ac.jp}
\maketitle

\begin{abstract}
In this paper, we consider a variant of Ramsey
numbers which we call complementary Ramsey numbers
$\bR(m,t,s)$.
We first establish their connections to pairs of Ramsey $(s,t)$-graphs.
Using the classification of Ramsey $(s,t)$-graphs for small $s,t$,
we determine the complementary
Ramsey numbers $\bR(m,t,s)$ for $(s,t)=(4,4)$ and $(3,6)$.
\end{abstract}

\section{Introduction}\label{sec:1}
For any given positive integers \(n_1,\ldots,n_c\), 
there is an integer, \(\bR\), such that if the edges of a 
complete graph of order at least \(\bR\) are colored with 
\(c\) different colors, then for some \(i\) between \(1\) and \(c\), 
there exists a complete subgraph of order \(n_i\) 
all of whose edges have colors different from \(i\).
The smallest such integer \(\bR\) is denoted by
\(\bR(n_1,\ldots,n_c)\).
Note that \(\bR(n_1,n_2)=R(n_2,n_1)\), an ordinary Ramsey number.
These numbers 
can be traced back (at least) to a paper of Erd\H{o}s, 
Hajnal and Rado \cite{EHS1965}. 
Erd\H{o}s and Szemeredi \cite{ErdosSzemeredi:1972}
proved that the diagonal complementary Ramsey number 
$\bR(n,\dots,n)$ (with $c$ colors), is at most $c^{C n/c}$, 
where $C$ is some constant.

Related concepts were considered later.
Chung and Liu \cite{MR523059} considered more general Ramsey numbers,
and the complementary Ramsey number
$\bR(m_1,\dots,m_t)$ is denoted by $R_{t-1}^t(K_{m_1},\dots,K_{m_t})$
in their notation.
Harborth and M\"oller \cite{MR1708843} also considered
what they call weakened Ramsey numbers $R_{s,t}(G)$. 
The relationship is
\[R_{t-1,t}(K_m)=\bR(\underbrace{m,\dots,m}_t).\]
Xu, Shao, Su, and Li \cite{GC2009}
considered multigraph Ramsey numbers, which are also regarded as 
more general than complementary Ramsey numbers. In fact,
the complementary Ramsey number
$\bR(m_1,\dots,m_t)$ is denoted by $f^{(t-1)}(m_1,\dots,m_t)$
in their notation.

For some small values of $m_1,m_2,m_3$, the complementary Ramsey numbers
$R(m_1,m_2,m_3)$ have been determined. We summarize these results in
Table~\ref{tab3}.
Note that we may assume without loss of generality, $m_1\geq m_2\geq m_3$.
The numbers $R(m,3,3)$, $R(m,4,3)$, and $R(m,5,3)$ have been
determined by 
Theorems~3.3, 3.4, and 3.5, respectively, in \cite{MR523059}.
Moreover, 
$\bR(4,4,4)=10$
by \cite[Theorem~3.6]{MR523059}.

\begin{table}
\begin{center}
\begin{tabular}{|c|cccccccc|}
\hline
$m$& $3$ & $4$ & $5$ & $6$ & $7$ & $8$ & $9$--$13$ & $14$-- \\
\hline 
$\bR(m,3,3)$ & {$5$} & $5$ & $5$ & {$6$} & $\cdots$ & $\cdots$ &$\cdots$ & $\cdots$  \\
\hline 
$\bR(m,4,3)$ & - & $7$ & $8$ & $8$ &  {$9$} & $\cdots$ & $\cdots$ & $\cdots$ \\
\hline
$\bR(m,5,3)$ & - & - & $9$ & $11$ &  {$12$} & $12$ & $13$ & $14$ \\
\hline
\end{tabular}
\caption{$\bR(m,m_2,3)$ for $m_2=3,4,5$}
\label{tab3}
\end{center}
\end{table}

The purpose of this paper is to determine
$\bR(m,4,4)$ and $\bR(m,6,3)$ for all positive integers $m$,
using the classification of Ramsey $(4,4)$-graphs and Ramsey
$(3,6)$-graphs, respectively, available from \cite{McK}.
Our results are tabulated in Tables~\ref{tab2x} and \ref{tab2y}.
The computation needed to verify the entries of these tables
was done with the help of Magma \cite{Magma}.

\begin{table}
\begin{center}
\begin{tabular}{|c|ccccccc|}
\hline
$m$& $4$ & $5$--$6$ & $7$ & $8$--$10$ & $11$--$16$ & $17$ & $18$--  \\
\hline 
$\bR(m,4,4)$ & $10$ & $13$ & $14$ & $15$ & $16$ & $17$ & $18$ \\
\hline
\end{tabular}
\caption{$\bR(m,4,4)$}
\label{tab2x}
\end{center}
\end{table}

\begin{table}
\begin{center}
\begin{tabular}{|c|ccccc|}
\hline
$m$& $6$ & $7$ & $8$ & $9$--$15$ & $16$--  \\
\hline 
$\bR(m,6,3)$ & $13$ & $14$ & $16$ & $17$ &$18$ \\
\hline
\end{tabular}
\caption{$\bR(m,6,3)$}
\label{tab2y}
\end{center}
\end{table}

This paper is organized as follows. In Sect.~\ref{sec:2},
we give definitions and derive immediate consequences. 
In Sect.~\ref{sec:rst}, we show how to determine the complementary
Ramsey number $\bR(m,t,s)$ from the knowledge of Ramsey $(s,t)$-graphs.
In Sect.~\ref{sec:44}, 
we determine
$\bR(m,4,4)$ and 
in Sect.~\ref{sec:36}, 
we determine
$\bR(m,6,3)$, for all positive integers $m$.
We end the article with concluding remarks 
as Sect.~\ref{sec:c}.

\section{Definitions and notation}\label{sec:2}
For a positive integer $n$, we denote the set $\{1,\dots,n\}$ by $[n]$,
and the set of all $n$-element subsets of a set $X$ by
$\binom{X}{n}$.
For positive integers $k$ and $n$, we denote the set of all edge-coloring
of the complete graph $K_n$ by $k$ colors, by $C(n,k)$:
\[
C(n,k)=\{f\mid f:\binom{[n]}{2}\to[k]\}.
\]
If $G$ is a graph, then we denote by $\alpha(G)$ the independence
number of $G$, and by $\omega(G)$ the clique number of $G$.
We identify a graph whose vertex set is $[n]$, with its set
of edges. In particular, for $f\in C(n,k)$ and $i\in[k]$,
$f^{-1}(i)$ is regarded as the graph $([n],f^{-1}(i))$. 
We use the abbreviations
$\alpha_i(f)=\alpha(f^{-1}(i))$,
$\omega_i(f)=\omega(f^{-1}(i))$.
The Ramsey number can be defined as follows:
\[
R(m_1,\dots,m_k)=\min\{n\in\N\mid
\forall f\in C(n,k),\;\exists i\in[k],\;\omega_i(f)\geq m_i\},
\]
where $m_1,\dots,m_k$ are positive integers.
The complementary Ramsey number is defined by replacing $\omega$
by $\alpha$ in the above definition of the Ramsey number:
\[
\bR(m_1,\dots,m_k)=\min\{n\in\N\mid
\forall f\in C(n,k),\;\exists i\in[k],\;\alpha_i(f)\geq m_i\}.
\]
The following properties are immediate from the
definition. In the following, $m_1,m_2,\dots$, denote positive integers.
\begin{align}
R(m_1,m_2)&=\bR(m_2,m_1),\notag\\
\bR(m_1,\dots,m_k)&=\bR(m_{\sigma(1)},\dots,m_{\sigma(k)})\label{e1}\\
&\qquad\text{for any permutation $\sigma$ on $[k]$,}\notag\\
\bR(m_1,\dots,m_k,2)&=\min\{m_1,\dots,m_k\}.\label{e4}
\end{align}
When considering $\bR(m_1,m_2,m_3)$, we may assume
$m_1\geq m_2\geq m_3\geq3$ without loss of generality, by
\eqref{e1} and \eqref{e4}.
We shall use an upper bound given in \cite[Theorem~2.1]{MR523059}:
\begin{equation}\label{2.1}
\bR(m_1,m_2,m_3)\leq R(m_2,m_3).
\end{equation}

\section{Some inequalities}
In this section, we establish some inequalities among complementary
Ramsey numbers. These inequalities will not be needed for our
main results, but it will be used to derive bounds for
$\bar{R}(5,5,5)$ in the final section.

Let $n$ and $k$ be positive integers. For $f\in C(n,k)$ and $x\in [n]$, set
\[
f_i(x)=|\{y\in[n]\mid f(\{x,y\})=i\}|.
\]

\begin{lemma}\label{lem:7}
Let $n,k,m_1,\dots,m_k$ and $t$ be positive integers with $1\leq t\leq k$, and
let $f\in C(n,k)$. If 
\[
\sum_{j=1}^t f_j(x)\geq 
\bar{R}(m_1,m_2,\ldots ,m_t , m_{t+1}-1,\ldots ,m_k-1) 
\]
for some $x\in [n]$, then
$\alpha_i(f)\geq m_i$ for some $i\in [k]$.
\end{lemma}
\begin{proof}
Set
$Y=\{y\in[n]\mid f(\{x,y\})\in[t]\}$,
so that
\[
|Y|=\sum_{j=1}^t f_j(x).
\]
Let $g=f|_{\binom{Y}{2}}$. 
By the assumption, either there exists $i$ with $1\leq i\leq t$
such that $\alpha_i(g)\geq m_i$, or there exists $i$ with
$t<i\leq k$ such that $\alpha_i(g)\geq m_i-1$.

If $1\leq i \leq t$, then $\alpha_i(f)\geq \alpha _i(g)\geq m_i$. 
If $t<i\leq k$, then there exists an independent set
$Z$ in $g^{-1}(i)$ with $|Z|=m_i-1$.
Then $Z\cup \{x\}$ is an independent set in $f^{-1}(i)$. 
This implies $\alpha_i(f)\geq m_i$. 
\end{proof}

\begin{lemma}\label{lem:8}
Let $k,m_1,\dots,m_k$ and $t$ be positive integers. 
Let $[k]=\bigcup_{j=1}^t M_j$ be a nontrivial partition.
For each $i\in[k]$ and $j\in[t]$, define
\[
m_i^{(j)}=\begin{cases}
m_i&\text{if $i \in M_j$,}\\
m_i-1&\text{otherwise.}
\end{cases}
\]
Then
\[
\bR(m_1, \ldots, m_k)\leq 
\sum_{j=1}^t\bR(m_1^{(j)},\dots,m_k^{(j)})-t+2.
\]
In particular,
\begin{equation}\label{m123}
\bR(m_1,m_2,m_3)\leq\bR(m_1,m_2,m_3-1)+\bR(m_1-1,m_2-1,m_3).
\end{equation}
\end{lemma}
\begin{proof}
Let $n$ denote the right-hand side of the inequality. 
If $f\in C(n,k)$ and $x\in [n]$, then
\begin{align*} 
\sum_{j=1}^t \sum_{i\in M_j} f_i(x)
&=\sum_{i=1}^k f_i(x)
\nexteq
n-1
\nexteq
\sum_{j=1}^t\bR(m_1^{(j)},\dots,m_k^{(j)})-t+1.
\\&>
\sum_{j=1}^t(\bR(m_1^{(j)},\dots,m_k^{(j)})-1).
\end{align*}
Thus, there exists $j\in[t]$ such that
\[
\sum_{i\in M_j} f_i(x)
\geq\bR(m_1^{(j)},\dots,m_k^{(j)}).
\]
By Lemma~\ref{lem:7}, there exists $i\in[k]$ such that
$\alpha_i(f)\geq m_i$. This implies
$\bar{R}(m_1,\dots,m_k)\leq n$.
\end{proof}

\section{Ramsey $(s,t)$-graphs}\label{sec:rst}

A graph $G$ is said to be a Ramsey $(s,t)$-graph if
$\omega(G)<s$ and $\alpha(G)<t$. 
We denote by $\cR_n(s,t)$
the set of Ramsey $(s,t)$-graphs on the vertex-set $[n]$.
Database of Ramsey graphs can be found in \cite{McK}.

We write $G\supseteq H$ if $H$ is a subgraph of $G$. 
Recall that, since we identify a graph with its set of edges, this means
that $G$ and $H$ have the same set of vertices, and the set of edges
of $H$ is a subset of that of $G$.
For a graph $G$ and its subgraph $H$, 
we denote by $G-H$ whose edge set consists of edges of $G$ which are
not an edge of $H$.


\begin{lemma}\label{lem:Rst}
Let $m_1,m_2,m_3$ be positive integers greater than $2$.
For a positive integer $n$, define 
\begin{align*}
a_n(m_3,m_2)&=\min\{\alpha(G-H)\mid G,H\in\cR_n(m_3,m_2),\; G\supseteq H\}.
\end{align*}
Then
\[\bR(m_1,m_2,m_3)=1+\max\{n\in\N\mid a_n(m_3,m_2)<m_1\}.\]
\end{lemma}
\begin{proof}
Define 
\begin{align*}
\cF_n&=\{f\mid f:\binom{[n]}{2}\to[3],\;\alpha(f^{-1}(i))<m_i\;(\forall i\in[3])\},\\
\cG_n&=\{(G,H)\mid G,H\in\cR_n(m_3,m_2),\;G\supseteq H,\; \alpha(G- H)<m_1\}.
\end{align*}
Then there is a bijection $\Phi:\cF_n\to\cG_n$ defined by
\[\Phi:f\mapsto(f^{-1}(\{1,2\}),f^{-1}(2))\quad(f\in\cF_n).\]
Indeed, for $f\in\cF_n$, write $G=f^{-1}(\{1,2\})$, $H=f^{-1}(2)$.
Then
\begin{align*}
m_1>\alpha(f^{-1}(1))&=\alpha(G- H),\\
m_2>\alpha(f^{-1}(2))&=\alpha(H),\\
m_3>\alpha(f^{-1}(3))&=\omega(G).
\end{align*}
Since $G\supseteq H$, we have $\omega(H)\leq \omega(G)$ and 
$\alpha(G)\leq\alpha(H)$. Thus, $G,H\in\cR_n(m_3,m_2)$, and 
$\Phi$ is well defined. It is clear that $\Phi$ is a bijection.
Now
\begin{align*}
\bR(m_1,m_2,m_3)-1&=
\max\{n\in\N\mid \cF_n\neq\emptyset\}
\nexteq
\max\{n\in\N\mid \cG_n\neq\emptyset\}
\nexteq
\max\{n\in\N\mid a_n(m_3,m_2)<m_1\}.
\end{align*}
\end{proof}


\section{The complementary Ramsey numbers $\bR(m,4,4)$}\label{sec:44}

Greenwood and Gleason \cite{GG} proved $R(4,4)=18$, which implies 
$\bR(m,4,4)\leq18$ for all $m\in\N$ by \eqref{2.1}.
The list of all Ramsey $(4,4)$-graphs can be found in \cite{McK}. 
For example,
$\cR_{17}(4,4)$ consists of a single graph,
while $\cR_{16}(4,4)$ consists of two graphs with the same size.
Thus $a_{n}(4,4)=n$ for $n=16,17$.
For $9\leq n\leq15$, we can use computer to determine
$a_n(4,4)$. Let us briefly describe how to perform this computation.
Observe that 
\begin{align*}
a_n(4,4)&=\min\{\alpha(G-H)\mid G,H\in\cR_n(4,4),\;G\supseteq H,\\
&\qquad\qquad G\text{ is maximal in }\cR_n(4,4)\}.
\end{align*}
Note that, by the same reason, we may further assume that $H$ is a minimal
Ramsey $(4,4)$-graph on $n$ vertices in the above definition of $a_n(4,4)$. 
However, enumerating all
containment relations between maximal and minimal 
Ramsey $(4,4)$-graphs on $n$ vertices could be difficult.
So we use a different approach. Fix a maximal 
Ramsey $(4,4)$-graph $G$ on $n$ vertices, and we try to find
a subgraph $H$ of $G$ such that $H\in\cR_n(4,4)$ and
$\alpha(G-H)\leq m$. This will guarantee
$a_n(4,4)\leq m$. 

If an edge $e\in G$ satisfies
$\alpha(G-e)\geq4$, then we must have $e\in H$. Indeed,
if $e\notin H$, then $G-e\supseteq H$, so $\alpha(G-e)\leq\alpha(H)\leq3$.
This is a contradiction.

If another edge $e'\in G$ satisfies $\alpha(G-e-e')>m$, then we must have $e'\notin H$.
Indeed, if $e'\in H$, then $G-e-e'\supseteq G-H$, so
$\alpha(G-e-e')\leq\alpha(G-H)\leq m$. This is a contradiction.

These criteria reduces the number of edges $e''$ for which
we need to determine whether $e''\in H$ or $e''\notin H$.
For such edges, we need to consider both possibilities, and eventually,
either we have a desired subgraph $H$, or the conclusion that $G$ has no
such subgraph $H$.

Even though the number of Ramsey $(4,4)$-graphs is quite large, 
the number of maximal Ramsey $(4,4)$-graphs of a given number of vertices
is small. We can quickly perform the above process for each of 
maximal Ramsey $(4,4)$-graphs to decide the truth of the inequality
$a_n(4,4)\leq m$. This leads to the determination of the values
$a_n(4,4)$ for all $n\in\{9,10,\dots,15\}$.
Our results are given in Table~\ref{tab2a}.

\begin{table}[h]
\begin{center}
\begin{tabular}{|c|cccccccc|}
\hline
$n$&        $9$ & $10$--$12$&$13$ &$14$ &$15$&$16$ &$17$ & $18$-- \\
\hline 
$a_n(4,4)$ & $3$ & $4$  & $6$ & $7$ & $10$ & $16$ & $17$ & $\infty$\\
\hline
\end{tabular}
\caption{$a_n(4,4)$}
\label{tab2a}
\end{center}
\end{table}

\begin{theorem}\label{thm:2x}
We have
\[\bR(m,4,4)=\begin{cases}
10&\text{if $m=4$,}\\
13&\text{if $m=5,6$,}\\
14&\text{if $m=7$,}\\
15&\text{if $8\leq m\leq10$,}\\
16&\text{if $11\leq m\leq16$,}\\
17&\text{if $m=17$,}\\
18&\text{if $m\geq18$.}
\end{cases}\]
\end{theorem}
\begin{proof}
Immediate from Table~\ref{tab2a} and Lemma~\ref{lem:Rst}.
\end{proof}

\section{The complementary Ramsey numbers $\bR(m,6,3)$}\label{sec:36}

K\'ery \cite{Kery} proved $R(3,6)=18$ (see also \cite{Car}), which implies 
$\bR(m,6,3)\leq18$ for all $m\in\N$ by \eqref{2.1}.
The list of all Ramsey $(3,6)$-graphs can be found in \cite{McK}. 
For $12\leq n\leq17$, we can use computer to determine
$a_n(3,6)$, exactly in the same manner as described in
Sect.~\ref{sec:44}. Our results are given in Table~\ref{tab2b}.

\begin{table}[h]
\begin{center}
\begin{tabular}{|c|cccccc|}
\hline
$n$&         
$12$ & $13$ & $14$--$15$ & $16$ & $17$ & $18$--\\
\hline 
$a_n(3,6)$ & 
$5$ &    $6$   & $7$    & $8$ &  $15$  &$\infty$\\
\hline
\end{tabular}
\caption{$a_n(3,6)$}
\label{tab2b}
\end{center}
\end{table}

\begin{theorem}\label{thm:2y}
We have
\[\bR(m,6,3)=\begin{cases}
13&\text{if $m=6$,}\\
14&\text{if $m=7$,}\\
16&\text{if $m=8$,}\\
17&\text{if $9\leq m\leq15$,}\\
18&\text{if $m\geq16$.}
\end{cases}\]
\end{theorem}
\begin{proof}
Immediate from Table~\ref{tab2b} and Lemma~\ref{lem:Rst}.
\end{proof}


\section{Concluding remarks}\label{sec:c}
If we could show $\bR(5,5,5)=12$, then
this would have contributed to the proof of a conjecture of Einhorn--Schoenberg
\cite{Shino}. This was our motivation to study complementary Ramsey numbers. 
However, Theorem~2.6 in \cite{MR523059} already gives
$\bR(5,5,5)\geq17$, which is improved in \cite[Table 2]{GC2009} as
$\bR(5,5,5)\geq20$.

As for an upper bound, observe first, by \eqref{m123}
and Tables~\ref{tab3}, \ref{tab2x}, we have
\begin{align*}
\bR(5,5,4)&\leq
\bR(5,5,3)+\bR(4,4,4)
\nexteq
19.
\end{align*}
Thus, again by \eqref{m123} and Table~\ref{tab2x},
we have
\begin{align*}
\bR(5,5,5)&\leq \bR(5,4,4)+\bR(4,5,5)
\\&\leq
32.
\end{align*}


\subsection*{Acknowledgements}
We would like to thank David Conlon and Xiao\-dong Xu
for bringing earlier literature to our attention.
The work of A.~M.\ was supported in part
by JSPS Open Partnership Joint Research Project 
``Extremal graph theory, algebraic graph theory and mathematical 
approach to network science'' (2017--18).
The work of M.~S.\ was supported by JSPS KAKENHI Grant Number 18K03396.

\end{document}